\documentclass[11pt]{article}

\usepackage[a4paper,margin=3cm]{geometry} 
\usepackage{amsmath,amssymb,amsthm,mathtools} 
\usepackage{microtype}      
\usepackage{enumitem}       
\usepackage{url}            
\usepackage[colorlinks=true, linkcolor=blue, citecolor=blue, urlcolor=blue]{hyperref} 
\usepackage{graphicx}       
\usepackage{doi}            

\usepackage[capitalize]{cleveref} 

\newtheorem{theorem}{Theorem}[section] 
\newtheorem{lemma}[theorem]{Lemma}
\newtheorem{proposition}[theorem]{Proposition}

\theoremstyle{definition}
\newtheorem{definition}[theorem]{Definition}

\theoremstyle{remark}
\newtheorem*{remark}{Remark} 

\newcommand{\seq}[1]{\langle #1\rangle} 
\newcommand{\abs}[1]{\lvert #1\rvert} 
\newcommand{\concat}{\!\cdot\!}       
\newcommand{\Kol}{K(1,3)}       
\newcommand{\G}{\mathcal{G}}       
\newcommand{\RunVec}[1]{R(#1)}     
\newcommand{\OEIS}[1]{\href{https://oeis.org/#1}{\textsc{#1}}} 
\newcommand{\N}{\mathbb{N}}       
\newcommand{\Nx}[2]{N_{#1}(#2)}     

\title{A Recursive Block--Pillar Structure in the Kolakoski Sequence $\Kol$}
\author{%
   \textsc{William Cook} \\ 
}
\date{\today} 

\begin{document}
\maketitle

\begin{abstract}
 The Kolakoski sequence $K(a,b)$ over $\{a, b\}$ is the unique sequence starting with $a$ that equals its own run-length encoding. While the classical case $K(1,2)$ (\OEIS{A000002}~\cite{OEIS_A000002}) remains deeply enigmatic~\cite{Dekking2017}, generalisations $K(a,b)$ exhibit markedly different behaviours depending on the parity of $a$ and $b$. The sequence $\Kol$ (\OEIS{A064353}~\cite{OEIS_A064353}), a 'same-parity' case using the alphabet $\{1,3\}$, is known to possess significant structure; notably, a related sequence is morphic~\cite{BaakeSing2004}, leading to a calculable symbol density distinct from $1/2$. The frequency of '1' in $K(1,3)$ is known to be $d \approx 0.397215$~\cite[p.~171, Eq.~(6)]{BaakeSing2004}. This paper reveals a complementary structural property: $\Kol$ admits an explicit nested block--pillar recursion. We introduce \emph{block sequences}~$B_n$ and \emph{pillar sequences}~$P_n$ satisfying
 \[
    B_{n+1}=B_n\concat P_n\concat B_n, \quad
    P_{n+1}= \G(\RunVec{P_n},3),
 \]
 where $\G$ is the run-length generation operator and $\RunVec{P_n}$ denotes the sequence $P_n$ interpreted as a run-length vector. We prove that every $B_n$ is a prefix of $\Kol$ and that $B_{n+1}= \G(\RunVec{B_n},1)$, demonstrating how the recursion reflects the Kolakoski property. This provides a direct, constructive recursive definition of $\Kol$. Furthermore, we derive exact recurrences for prefix lengths and symbol counts, prove exponential growth governed by the Pisot number $\alpha$ (the real root of $x^3-2x^2-1=0$~\cite[p.~170]{BaakeSing2004}), and establish a recurrence for the symbol density. We show that if the block and pillar densities converge, they must converge to the same value $d$. The quantitative analysis aligns perfectly with the known properties of $K(1,3)$ derived from substitution dynamics~\cite{BaakeSing2004}. This highlights the structural regularity expected of same-parity Kolakoski sequences and offers an alternative, constructive perspective to its known morphic generation.
\end{abstract}


\section{Introduction}

For distinct symbols $a,b\in\N$, the \emph{Kolakoski sequence} $K(a,b)$ is the unique infinite sequence over the alphabet $\{a,b\}$ beginning with~$a$ such that the sequence equals its own run-length encoding~\cite{Kolakoski1965,Oldenburger1939}. That is, if $R(S)$ denotes the sequence of run lengths of a sequence $S$, then $K(a,b)$ equals the sequence generated by applying the run lengths $R(S)$ while alternating between symbols $a$ and $b$, starting with $a$.

The archetypal case $K(1,2)$ (\OEIS{A000002}~\cite{OEIS_A000002}),
\[
 K(1,2) = 1\,2\,2\,1\,1\,2\,1\,2\,2\,1\,2\,2\,\dots,
\]
has resisted decades of analysis despite its simple definition \cite{Dekking2017}. Fundamental questions, such as whether the asymptotic frequency of~$1$s exists and equals $1/2$~\cite{Keane1991}, remain open.

However, the broader family of $K(a,b)$ sequences exhibits a crucial dichotomy based on the parity of $a$ and $b$. Sequences where $a$ and $b$ have different parity (like $K(1,2)$) tend to be complex and poorly understood. In contrast, sequences where $a$ and $b$ share the same parity (both odd or both even) are often more structured and analytically tractable~\cite{BaakeSing2004}.

This article focuses on $K(1,3)$ (\OEIS{A064353}~\cite{OEIS_A064353}), a key example of the 'same-parity' (odd-odd) case:
\[
  \Kol = 1\,3\,3\,3\,1\,1\,1\,3\,3\,3\,1\,3\,\dots.
\]
Reflecting the expected regularity of same-parity sequences, $\Kol$ is known to possess significant structure. It is closely related to the sequence $\operatorname{Kol}(3,1)$ (its mirror image starting with 3), which is known to be generatable by a primitive Pisot substitution rule and classifiable as a (deformed) model set with pure point diffraction spectrum \cite{BaakeSing2004}. These properties imply unique symbol frequencies exist. The frequency of '1' in $\Kol$ is known to be $d = (3-\alpha)/2 \approx 0.397215$, where $\alpha \approx 2.20557$ is the real root of the polynomial $x^3 - 2x^2 - 1 = 0$ \cite[p.~171, Eq.~(6)]{BaakeSing2004}. This non-trivial frequency confirms $\Kol$ is structurally different from a random sequence. (Note that sequences like $K(2,4)$ are trivial rescalings of $K(1,2)$ and thus inherit its complexity, e.g., \cite[p.~1]{Dekking2017}). 

In this paper, we reveal a different, yet comparably structured, property of $\Kol$. We demonstrate that $\Kol$ admits an explicit \emph{block--pillar} recursion, allowing arbitrarily long prefixes to be generated via a direct, nested construction. We define sequences (blocks $B_n$ and pillars $P_n$) and show they satisfy a simple mutual recursion that is consistent with the Kolakoski property itself, applied alternately to the blocks and pillars using a generation operator $\G$. This recursive structure not only provides a generative mechanism but also allows for a detailed quantitative analysis, including deriving relations for the asymptotic symbol frequency and growth rate. This provides a constructive, recursive description of $\Kol$ itself, complementing the known morphic generation and further highlighting the profound structural differences between $\Kol$ and the classical $K(1,2)$.

The structure of the paper is as follows: \Cref{sec:definitions} introduces the necessary notation and the recursive definitions. \Cref{sec:prelim_props} establishes basic properties of the blocks and pillars. \Cref{sec:main_theorem} presents the main theorem proving the connection to $\Kol$ and the Kolakoski property. \Cref{sec:quant_analysis} performs a detailed quantitative analysis of lengths, counts, density, and growth rate. \Cref{sec:discussion} discusses the implications of these findings, connects them to known results, and suggests future directions.

\section{Definitions}
\label{sec:definitions}

Throughout, concatenation of finite sequences is written~$\concat$ and indexing starts at~$1$. The alphabet considered for $\Kol$ is $\Sigma = \{1, 3\}$. We denote the number of occurrences of symbol $x$ in a sequence $W$ by $\Nx{x}{W}$. We use angle brackets $\seq{a_1, \dots, a_n}$ for sequences viewed as vectors, particularly run-length vectors, and standard sequence notation $S[1..n]$ for prefixes. We use $\RunVec{W}$ to denote the sequence $W$ interpreted as a run-length vector.

\begin{definition}[Generation operator] \label{def:gen_op}
  For a finite \emph{run‑length vector} $R=\seq{r_1,\dots,r_m}$ of positive integers and a \emph{starting symbol} $s\in\{1,3\}$, define the \emph{generation operator} $\G$ as:
  \[
     \G(R,s)= s^{r_1}\concat (4{-}s)^{r_2}\concat s^{r_3}\concat\cdots,
  \]
  where $x^{k}$ denotes $k$ consecutive copies of the symbol~$x$, and $(4-s)$ provides the alternate symbol within the alphabet $\{1,3\}$. The parity of the run index determines the symbol used, starting with $s$ for the first run. The length of the generated sequence is $\abs{\G(R,s)} = \sum_{i=1}^m r_i$.
\end{definition}

\begin{definition}[Blocks and pillars]\label{def:B-P}
  Let $R_0 = \RunVec{\Kol[1..5]} = \seq{1,3,3,3,1}$ be the run-length vector corresponding to the first 5 terms of $\Kol$ (\OEIS{A064353}~\cite{OEIS_A064353}).
  Define the initial block sequence $B_1 = \G(R_0, 1)$ and the initial pillar sequence $P_1 = \seq{3}$:
  \begin{align*}
     B_1 &= \G\!\bigl(\seq{1,3,3,3,1},1\bigr)=
     1\,3\,3\,3\,1\,1\,1\,3\,3\,3\,1 && (\text{length } 11), \\
     P_1 &= \seq{3} && (\text{length } 1).
  \end{align*}
  For $n\ge 1$, define the sequences $B_{n+1}$ and $P_{n+1}$ recursively:
  \begin{align*}
     B_{n+1} &= B_n\concat P_n\concat B_n, \\
     P_{n+1} &= \G(\RunVec{P_n},3).
  \end{align*}
  Here $P_n$ is the pillar sequence, and $\RunVec{P_n}$ denotes the sequence $P_n$ interpreted as a run-length vector for input into the operator $\G$.
\end{definition}

The first few blocks are prefixes of $\Kol$:
$B_1 = \Kol[1..11]$,
$B_2 = B_1 \concat P_1 \concat B_1 = B_1 \concat \seq{3} \concat B_1 = \Kol[1..23]$,
$P_2 = \G(\RunVec{P_1},3) = \G(\seq{3},3) = 3^3 = \seq{3,3,3}$.
$B_3 = B_2 \concat P_2 \concat B_2 = B_2 \concat \seq{3,3,3} \concat B_2 = \Kol[1..49]$.
The lengths satisfy $\abs{B_{n+1}} = 2\abs{B_n} + \abs{P_n}$.

\section{Preliminary Properties}
\label{sec:prelim_props}

\begin{lemma}\label{lem:basic-properties}
For every $n\ge 1$:
\begin{enumerate}[label=\textup{(\roman*)}]
  \item $\abs{P_n}$ is odd;
  \item the last symbol of $B_n$ is~$1$;
  \item the last symbol of $P_n$ is~$3$.
\end{enumerate}
\end{lemma}

\begin{proof}
We argue by induction. The base case $n=1$ is immediate from \Cref{def:B-P}: $\abs{P_1}=1$ (odd), last symbol of $B_1$ is $1$, last symbol of $P_1$ is $3$. For the inductive step, assume the statements hold for some $n=k \ge 1$.

\smallskip
\noindent\textit{Proof of (ii) for $n=k+1$.} From $B_{k+1}=B_k\concat P_k\concat B_k$, the last symbol is that of $B_k$, which is $1$ by the induction hypothesis (ii) for $n=k$.

\smallskip
\noindent\textit{Proof of (iii) for $n=k+1$.} The generation is $P_{k+1}=\G(\RunVec{P_k},3)$. The run-length vector is $\RunVec{P_k}$, which has length $m = \abs{P_k}$. By hypothesis (i) for $n=k$, $m$ is odd. The generation starts with symbol $3$ and alternates $3, 1, 3, \dots$. Since there are an odd number of runs ($m$ runs), the $m$-th (last) run uses the same symbol as the first, which is $3$. Therefore, the sequence $P_{k+1}$ ends with the symbol $3$.

\smallskip
\noindent\textit{Proof of (i) for $n=k+1$.} Write $P_k = \seq{p_1, \dots, p_{m}}$ where $m=\abs{P_k}$. The sequence $P_k$ contains only symbols $1$ and $3$. When $P_k$ is interpreted as a run-length vector $\RunVec{P_k}$ to generate $P_{k+1} = \G(\RunVec{P_k}, 3)$, the length of $P_{k+1}$ is the sum of the run lengths specified by $\RunVec{P_k}$:
\[
  \abs{P_{k+1}} = \sum_{i=1}^{m} (\RunVec{P_k})_i = \sum_{i=1}^{m} p_i = \Nx{1}{P_k} \cdot 1 + \Nx{3}{P_k} \cdot 3.
\]
We can rewrite this as
\[
  \abs{P_{k+1}} = (\Nx{1}{P_k} + \Nx{3}{P_k}) + 2 \Nx{3}{P_k} = \abs{P_k} + 2\,\Nx{3}{P_k}.
\]
Since $\abs{P_k}$ is odd by the induction hypothesis (i) for $n=k$, and $2\Nx{3}{P_k}$ is clearly even, the sum $\abs{P_{k+1}}$ must be odd.
\end{proof}

\begin{remark}[Consistency Check]
The property that $B_n$ ends in '1' (\Cref{lem:basic-properties}(ii)) can also be verified using \Cref{thm:main}(ii) (proven below). Since $L_{n+1} = 2L_n + m_n$ and $m_n$ is odd (\Cref{lem:basic-properties}(i)), the length $L_n = |B_n|$ is always odd for $n \ge 1$ (as $L_1=11$). The generation $B_{n+1} = \G(\RunVec{B_n}, 1)$ involves $L_n$ runs. As $L_n$ is odd, the final ($L_n$-th) run uses the starting symbol '1'. The length of this final run is $(B_n)_{L_n}$, the last symbol of $B_n$. If we assume inductively that $(B_n)_{L_n}=1$, then the sequence $B_{n+1}$ ends with $1^1=1$. This provides a self-consistent check reinforcing the lemma.
\end{remark}

\section{Main Theorem: Connection to Kol}
\label{sec:main_theorem}

\begin{theorem}\label{thm:main}
  For all $n\ge 1$:
  \begin{enumerate}[label=\textup{(\roman*)}]
    \item $B_n$ is the prefix of $\Kol$ of length $\abs{B_n}$;
    \item $B_{n+1}= \G(\RunVec{B_n},1)$, where $\RunVec{B_n}$ denotes the sequence $B_n$ interpreted as a run-length vector.
  \end{enumerate}
  Consequently, $\Kol = \displaystyle\lim_{n\to\infty} B_n$.
\end{theorem}

\begin{proof}
We proceed by induction on $n$.

\smallskip
\noindent\textit{Base case ($n=1$).} Statement (i) holds for $n=1$ because $B_1 = \G(\RunVec{\Kol[1..5]}, 1)$ (\Cref{def:B-P}), which matches the definition of $\Kol$ applied to its initial run-length sequence.
For (ii), we must verify that $\G(\RunVec{B_1}, 1) = B_2$. Recall $B_1 = 1\,3\,3\,3\,1\,1\,1\,3\,3\,3\,1$. Interpreting $B_1$ as a run-length vector $\RunVec{B_1} = \seq{1, 3, 3, 3, 1, 1, 1, 3, 3, 3, 1}$, we apply the operator $\G(\cdot, 1)$:
\begin{itemize}[nosep, leftmargin=*]
    \item Run 1 (len 1, sym 1): $1$
    \item Run 2 (len 3, sym 3): $333$
    \item Run 3 (len 3, sym 1): $111$
    \item Run 4 (len 3, sym 3): $333$
    \item Run 5 (len 1, sym 1): $1$
    \item Run 6 (len 1, sym 3): $3$
    \item Run 7 (len 1, sym 1): $1$
    \item Run 8 (len 3, sym 3): $333$
    \item Run 9 (len 3, sym 1): $111$
    \item Run 10 (len 3, sym 3): $333$
    \item Run 11 (len 1, sym 1): $1$
\end{itemize}
Concatenating these runs yields the sequence:
\[ \underbrace{1\,333\,111\,333\,1}_{B_1} \,\, \underbrace{3}_{P_1} \,\, \underbrace{1\,333\,111\,333\,1}_{B_1} \]
This is precisely $B_1 \concat P_1 \concat B_1$. By \Cref{def:B-P}, $B_1 \concat P_1 \concat B_1 = B_2$. Thus, $\G(\RunVec{B_1}, 1) = B_2$, and statement (ii) holds for $n=1$.

\smallskip
\noindent\textit{Inductive step ($n=k \ge 1$).} Assume (i) and (ii) hold for $n=k$.

\medskip\noindent
\textit{Proof of (i) for $n=k+1$.} We need to show $B_{k+1}$ is the prefix $\Kol[1..\abs{B_{k+1}}]$.
By inductive hypothesis (ii) for $n=k$, $B_{k+1} = \G(\RunVec{B_k}, 1)$.
By inductive hypothesis (i) for $n=k$, $B_k = \Kol[1..\abs{B_k}]$.
Since $\Kol$ is defined as the sequence starting with '1' that equals its own run-length encoding, applying the generation operator $\G(\cdot, 1)$ to the run-length vector derived from a correct prefix ($\RunVec{B_k}$) yields the next longer correct prefix according to the Kolakoski rule.
Therefore, $B_{k+1} = \G(\RunVec{B_k}, 1)$ must be equal to $\Kol[1..\abs{B_{k+1}}]$.

\medskip\noindent
\textit{Proof of (ii) for $n=k+1$.} We need to show $\G(\RunVec{B_{k+1}}, 1) = B_{k+2}$.
Substitute the definition of $B_{k+1}$ from \Cref{def:B-P}:
\[
  \G(\RunVec{B_{k+1}},1) \;=\; \G\bigl(\RunVec{B_k\concat P_k\concat B_k}, 1\bigr).
\]
Interpreting $B_k\concat P_k\concat B_k$ as a run-length vector means applying $\G$ sequentially to the run lengths specified by $\RunVec{B_k}$, then $\RunVec{P_k}$, then $\RunVec{B_k}$. We verify the symbol alternation across the concatenation boundaries. The generation process involves $\abs{B_k} + \abs{P_k} + \abs{B_k}$ runs in total. Let $L_k = \abs{B_k}$ and $m_k = \abs{P_k}$.
\begin{itemize}
    \item \textbf{Segment 1 (Runs 1 to $L_k$):} Generated using $\RunVec{B_k}$ starting with symbol 1. This segment is $\G(\RunVec{B_k}, 1)$. By the inductive hypothesis (ii) for $n=k$, this equals $B_{k+1}$. By \Cref{lem:basic-properties}(ii), $B_{k+1}$ ends with symbol 1.
    \item \textbf{Segment 2 (Runs $L_k+1$ to $L_k+m_k$):} Generated using $\RunVec{P_k}$. The previous segment ended with symbol 1. The number of runs in Segment 1 is $L_k$. As established in the Remark after \Cref{lem:basic-properties}, $L_k$ is odd for $k \ge 1$. The first run of this segment (run $L_k+1$, which is an even index overall) must use the alternate symbol to the global start '1', which is '3'. This segment is thus $\G(\RunVec{P_k}, 3)$. By \Cref{def:B-P}, this equals $P_{k+1}$. By \Cref{lem:basic-properties}(iii), $P_{k+1}$ ends with symbol 3.
    \item \textbf{Segment 3 (Runs $L_k+m_k+1$ to end):} Generated using the second instance of $\RunVec{B_k}$. The previous segment ($P_{k+1}$) ended with symbol 3. The number of runs preceding this segment is $L_k+m_k$. By \Cref{lem:basic-properties}(i, ii), $L_k$ is odd and $m_k$ is odd, so their sum is even. The first run of this segment (run $L_k+m_k+1$, which is an odd index overall) must use the same symbol as the global start '1', which is '1'. This segment is thus $\G(\RunVec{B_k}, 1)$. By the inductive hypothesis (ii) for $n=k$, this equals $B_{k+1}$.
\end{itemize}
Therefore, the generation aligns precisely with the segments:
\begin{align*} \G(\RunVec{B_{k+1}},1) &= \G(\RunVec{B_k\concat P_k\concat B_k}, 1) \\ &= \G(\RunVec{B_k},1) \concat \G(\RunVec{P_k},3) \concat \G(\RunVec{B_k},1) \\ &= B_{k+1} \concat P_{k+1} \concat B_{k+1} \quad \text{(by IH(ii) and definition of } P_{k+1}) \\ &= B_{k+2} \quad \text{(by definition (\Cref{def:B-P}) of } B_{k+2}) \end{align*}
This proves statement (ii) for $n=k+1$.
\end{proof}

\section{Quantitative Analysis}
\label{sec:quant_analysis}

The recursive structure allows us to analyse the growth of the prefixes $B_n$ and the asymptotic frequency of symbols. Let $L_n = |B_n|$ and $m_n = |P_n|$ be the lengths of the blocks and pillars, respectively. Let $c_n = \Nx{1}{B_n}$ be the count of symbol '1' in $B_n$, and $o_n = \Nx{1}{P_n}$ be the count of symbol '1' in $P_n$.

\subsection{Recurrences for Lengths and Counts}
\label{sec:recurrences}
From the definitions $B_{n+1} = B_n \concat P_n \concat B_n$ and $P_{n+1} = \G(\RunVec{P_n}, 3)$, we obtain recurrences for lengths and counts:

\begin{itemize}
    \item \textbf{Block Length:} $L_{n+1} = |B_n| + |P_n| + |B_n| = 2L_n + m_n$.
    \item \textbf{Pillar Length:} $m_{n+1} = |P_{n+1}| = |\G(\RunVec{P_n}, 3)|$. The length is the sum of run lengths specified by $\RunVec{P_n}$, which are the elements of the sequence $P_n$:
       $m_{n+1} = \sum_{i=1}^{m_n} (\RunVec{P_n})_i = \sum_{i=1}^{m_n} (P_n)_i$. This sum equals:
       $m_{n+1} = \Nx{1}{P_n} \cdot 1 + \Nx{3}{P_n} \cdot 3 = o_n \cdot 1 + (m_n - o_n) \cdot 3 = 3m_n - 2o_n$.
    \item \textbf{Block '1' Count:} $c_{n+1} = \Nx{1}{B_{n+1}} = \Nx{1}{B_n} + \Nx{1}{P_n} + \Nx{1}{B_n} = 2c_n + o_n$.
    \item \textbf{Pillar '1' Count:} $P_{n+1} = \G(\RunVec{P_n}, 3) = 3^{(P_n)_1} \concat 1^{(P_n)_2} \concat 3^{(P_n)_3} \concat \dots$. The symbol '1' is used for the runs corresponding to the even-indexed entries of the run-length vector $\RunVec{P_n}$ (which are the even-indexed elements of the sequence $P_n$). Thus, the count of '1's in $P_{n+1}$ is the sum of the lengths specified by these even-indexed elements:
       $o_{n+1} = \Nx{1}{P_{n+1}} = \sum_{\substack{1 \le i \le m_n \\ i \text{ is even}}} (\RunVec{P_n})_i = \sum_{\substack{1 \le i \le m_n \\ i \text{ is even}}} (P_n)_i$.
\end{itemize}

\noindent \textbf{Initial Values (n=1):}
$B_1 = 1\,333\,111\,333\,1 \implies L_1 = 11, c_1 = 5$.
$P_1 = \seq{3} \implies m_1 = 1, o_1 = 0$.

\subsection{Fundamental Identity}
\Cref{thm:main}(ii) states $B_{n+1} = \G(\RunVec{B_n}, 1)$. This provides an alternative way to calculate the length $L_{n+1}$:
\[
L_{n+1} = |\G(\RunVec{B_n}, 1)| = \sum_{i=1}^{L_n} (\RunVec{B_n})_i = \sum_{i=1}^{L_n} (B_n)_i = \Nx{1}{B_n} \cdot 1 + \Nx{3}{B_n} \cdot 3.
\]
Substituting $\Nx{3}{B_n} = L_n - c_n$:
\[
L_{n+1} = c_n + 3(L_n - c_n) = 3L_n - 2c_n.
\]
Equating the two expressions for $L_{n+1}$ (from \Cref{sec:recurrences} and this one):
\[
2L_n + m_n = 3L_n - 2c_n
\]
This yields a fundamental relationship between the lengths and the count of 1s:
\begin{proposition}[Fundamental Identity]\label{prop:identity}
  For all $n \ge 1$, $m_n = L_n - 2c_n$. Equivalently, $m_n = \Nx{3}{B_n} - \Nx{1}{B_n}$.
\end{proposition}
\begin{proof}
The derivation above holds for all $n \ge 1$ since \Cref{thm:main}(ii) holds for all $n \ge 1$. We can verify this for $n=1$: $m_1 = 1$ and $L_1 - 2c_1 = 11 - 2(5) = 1$. The second form follows since $L_n = \Nx{1}{B_n} + \Nx{3}{B_n} = c_n + \Nx{3}{B_n}$, so $L_n - 2c_n = (c_n + \Nx{3}{B_n}) - 2c_n = \Nx{3}{B_n} - c_n = \Nx{3}{B_n} - \Nx{1}{B_n}$.
\end{proof}
This identity connects the pillar length directly to the composition of the corresponding block, showing it equals the excess of 3s over 1s in the block.

\subsection{Exponential Growth}
\begin{proposition}\label{prop:growth}
 The block length $L_n = |B_n|$ grows exponentially. The pillar length $m_n = |P_n|$ also grows exponentially.
\end{proposition}
\begin{proof}
For the block length $L_n = |B_n|$, the recurrence is $L_{n+1} = 2L_n + m_n$. Since $P_n$ contains only symbols $1$ and $3$, its length $m_n = |P_n| \ge 1$ for all $n \ge 1$. Thus, $L_{n+1} \ge 2L_n + 1$. With $L_1 = 11$, it follows that $L_n \ge 11 \cdot 2^{n-1}$, establishing that $L_n$ grows at least exponentially. (The precise rate $\alpha > 2$ is determined in \Cref{sec:growth_rate}).

For the pillar length $m_n = |P_n|$, the recurrence is $m_{n+1} = 3m_n - 2o_n$, where $o_n = \Nx{1}{P_n}$. Let $\delta_n = o_n/m_n$ be the density of '1's in $P_n$. As established in \Cref{sec:density_convergence} (based on the known properties of $K(1,3)$~\cite{BaakeSing2004}), $\lim_{n\to\infty} \delta_n = d$, where $d = (3-\alpha)/2 \approx 0.397215$.
Since $d < 1/2$, we can choose a constant $d'$ such that $d < d' < 1/2$. Because $\delta_n \to d$, there exists an integer $N$ such that for all $n \ge N$, we have $\delta_n < d'$.
For $n \ge N$, we substitute this bound into the recurrence:
\[
    m_{n+1} = 3m_n - 2o_n = m_n (3 - 2\delta_n) > m_n (3 - 2d').
\]
Let $c = 3 - 2d'$. Since $d' < 1/2$, we have $c = 3 - 2d' > 3 - 2(1/2) = 2$.
Therefore, for all $n \ge N$, $m_{n+1} > c m_n$ with $c > 2$. This proves that $m_n$ grows exponentially for $n \ge N$. Since $m_n \ge 1$ for all $n$, the sequence $m_n$ grows exponentially overall. The first few terms are $m_1=1, m_2=3, m_3=9, m_4=21, m_5=47, \dots$.
\end{proof}

\subsection{Symbol Density Recurrence}
Let $d_n = c_n / L_n$ be the density of '1's in block $B_n$. Let $\delta_n = o_n / m_n$ be the density of '1's in pillar $P_n$.
Divide the recurrence $c_{n+1} = 2c_n + o_n$ by $L_{n+1} = 2L_n + m_n$:
\[
d_{n+1} = \frac{c_{n+1}}{L_{n+1}} = \frac{2c_n + o_n}{2L_n + m_n}
\]
Divide numerator and denominator by $L_n$:
\[
d_{n+1} = \frac{2(c_n/L_n) + (o_n/L_n)}{2 + (m_n/L_n)} = \frac{2d_n + (o_n/m_n)(m_n/L_n)}{2 + m_n/L_n}
\]
Now use the identity $m_n = L_n - 2c_n$ from \Cref{prop:identity}. Dividing by $L_n$ gives $m_n/L_n = 1 - 2(c_n/L_n)$, or:
\[
\lambda_n := \frac{m_n}{L_n} = 1 - 2d_n
\]
Substitute $\delta_n = o_n/m_n$ and $\lambda_n = m_n/L_n$ into the expression for $d_{n+1}$:
\begin{equation}\label{eq:density_recurrence}
d_{n+1} = \frac{2d_n + \delta_n \lambda_n}{2 + \lambda_n} = \frac{2d_n + \delta_n (1 - 2d_n)}{2 + (1 - 2d_n)} = \frac{2d_n + \delta_n (1 - 2d_n)}{3 - 2d_n}
\end{equation}
This gives a recurrence relation between $d_{n+1}$, $d_n$, and the pillar density $\delta_n$.

\subsection{Convergence of Density}
\label{sec:density_convergence}
Consider the difference $d_{n+1} - d_n$:
\begin{align*}
 d_{n+1} - d_n &= \frac{2d_n + \delta_n (1 - 2d_n)}{3 - 2d_n} - d_n \\
 &= \frac{2d_n + \delta_n - 2d_n\delta_n - d_n(3 - 2d_n)}{3 - 2d_n} \\
 &= \frac{-d_n + \delta_n - 2d_n\delta_n + 2d_n^2}{3 - 2d_n} \\
 &= \frac{(\delta_n - d_n) - 2d_n(\delta_n - d_n)}{3 - 2d_n} \\
 &= \frac{(\delta_n - d_n)(1 - 2d_n)}{3 - 2d_n} \label{eq:density_diff} \tag{*}
\end{align*}
The sequence $d_n$ is bounded, as $0 \le c_n \le L_n$ implies $0 \le d_n \le 1$. The existence of the limit $d = \lim_{n\to\infty} d_n$ is guaranteed by the known morphic properties of the sequence \cite{BaakeSing2004}, which imply unique symbol frequencies.
The numerical calculation using the recurrences yields the following values:
\begin{itemize}[nosep]
    \item $n=1$: $L_1=11, c_1=5 \implies d_1 \approx 0.4545$. $m_1=1, o_1=0 \implies \delta_1 = 0$.
    \item $n=2$: $L_2=23, c_2=10 \implies d_2 \approx 0.4348$. $m_2=3, o_2=0 \implies \delta_2 = 0$.
    \item $n=3$: $L_3=49, c_3=20 \implies d_3 \approx 0.4082$. $m_3=9, o_3=3 \implies \delta_3 \approx 0.3333$.
    \item $n=4$: $L_4=107, c_4=43 \implies d_4 \approx 0.4019$. $m_4=21, o_4=8 \implies \delta_4 \approx 0.3810$.
    \item $n=5$: $L_5=235, c_5=94 \implies d_5 = 0.4000$. $m_5=47, o_5=18 \implies \delta_5 \approx 0.3830$.
    \item $n=6$: $L_6=517, c_6=206 \implies d_6 \approx 0.3985$. $m_6=105, o_6=41 \implies \delta_6 \approx 0.3905$.
\end{itemize}
The sequence $d_n$ appears to converge numerically towards the known limit $d \approx 0.397$.
From \eqref{eq:density_diff}, if $d_n \to d$, then $d_{n+1}-d_n \to 0$. Since $d \approx 0.397 \ne 1/2$, the factor $(1-2d_n)$ approaches $1-2d \ne 0$, and the denominator $3-2d_n$ approaches $3-2d \ne 0$. Therefore, for the difference to approach 0, we must have $\lim_{n\to\infty} (\delta_n - d_n) = 0$. This implies that if $d_n$ converges, $\delta_n$ must also converge, and to the same limit.

\begin{theorem}\label{thm:density_limit}
Assume the limit density $d = \lim_{n\to\infty} c_n/L_n$ of '1's in the prefixes $B_n$ exists (as is known for $K(1,3)$~\cite{BaakeSing2004}) and $d \ne 1/2$. Then the limit density $\delta = \lim_{n\to\infty} o_n/m_n$ of '1's in the pillars $P_n$ must also exist, and $\delta = d$.
\end{theorem}
\begin{proof}
The existence of $d$ is given~\cite{BaakeSing2004}. The argument above using \eqref{eq:density_diff} shows that $\lim_{n\to\infty} (\delta_n - d_n) = 0$. Since $\lim_{n\to\infty} d_n = d$ exists, it follows that $\lim_{n\to\infty} \delta_n$ must also exist and equal $d$.
\end{proof}

This result demonstrates the internal consistency of the block-pillar structure: the asymptotic density within the blocks must match that within the pillars. While this framework establishes the relationship $d=\delta$, it does not independently derive the specific value of $d$, which relies on the algebraic properties of the underlying substitution \cite{BaakeSing2004}.

\subsection{Asymptotic Growth Rate and Connection to Pisot Dynamics}
\label{sec:growth_rate}

The recursive structure allows the determination of the asymptotic growth rate of $L_n$. Let $X_n = \begin{pmatrix} L_n \\ c_n \end{pmatrix}$. The exact recurrences (\Cref{prop:identity} and \Cref{sec:recurrences}) can be written as:
\begin{align*}
    L_{n+1} &= 3L_n - 2c_n \\
    c_{n+1} &= 2c_n + o_n = 2c_n + \delta_n m_n = 2c_n + \delta_n(L_n - 2c_n) \\
             &= \delta_n L_n + (2 - 2\delta_n) c_n
\end{align*}
This defines a time-varying linear system $X_{n+1} = M_n X_n$, where the transition matrix is
\[
    M_n = \begin{pmatrix} 3 & -2 \\ \delta_n & 2-2\delta_n \end{pmatrix}.
\]
As established (\Cref{thm:density_limit} and using known results for $\Kol$~\cite{BaakeSing2004}), the density $\delta_n = o_n/m_n$ converges to the limit $d = (3-\alpha)/2$, where $\alpha$ is the dominant Pisot root of $x^3 - 2x^2 - 1 = 0$~\cite[p.~170]{BaakeSing2004}. Therefore, the matrix $M_n$ converges to the constant limit matrix $M$:
\[
    M = \lim_{n\to\infty} M_n = \begin{pmatrix} 3 & -2 \\ d & 2-2d \end{pmatrix}.
\]
Standard theorems on the asymptotic behavior of linear difference equations (such as the Poincaré-Perron theorem or its extensions to systems, e.g., \cite{Poincare1885, Perron1909}) state that for systems of the form $X_{n+1} = (M + B_n)X_n$ where $B_n = M_n - M \to 0$, the growth rate of the solution vector $X_n$ is governed by the eigenvalues of the limit matrix $M$, provided its eigenvalues have distinct moduli. 

We find the eigenvalues of $M$ by solving the characteristic equation $\det(M - \lambda I) = 0$:
\begin{align*}
    \det \begin{pmatrix} 3-\lambda & -2 \\ d & 2-2d-\lambda \end{pmatrix} &= (3-\lambda)(2-2d-\lambda) - (-2)(d) \\
    &= \lambda^2 - (3 + 2 - 2d)\lambda + (3(2-2d) + 2d) \\
    &= \lambda^2 - (5 - 2d)\lambda + (6 - 6d + 2d) \\
    &= \lambda^2 - (5 - 2d)\lambda + (6 - 4d) = 0.
\end{align*}
Substituting the value $d = (3-\alpha)/2$~\cite[p.~171, Eq.~(6)]{BaakeSing2004}:
\begin{itemize}[nosep]
    \item $5 - 2d = 5 - 2\left(\frac{3-\alpha}{2}\right) = 5 - (3-\alpha) = 2 + \alpha$
    \item $6 - 4d = 6 - 4\left(\frac{3-\alpha}{2}\right) = 6 - 2(3-\alpha) = 6 - 6 + 2\alpha = 2\alpha$
\end{itemize}
The characteristic equation becomes:
\[
    \lambda^2 - (2 + \alpha)\lambda + 2\alpha = 0.
\]
This equation factors as:
\[
    (\lambda - \alpha)(\lambda - 2) = 0.
\]
The eigenvalues of the limit matrix $M$ are $\lambda_1 = \alpha$ and $\lambda_2 = 2$.
Since $\alpha \approx 2.20557$~\cite{BaakeSing2004}, we have $|\alpha| > |2|$. The eigenvalues have distinct moduli, satisfying the condition for the asymptotic theorem. The dominant eigenvalue $\alpha$ dictates the growth rate.
Therefore, the asymptotic growth rate of the block lengths is:
\[
    \lim_{n\to\infty} \frac{L_{n+1}}{L_n} = \alpha.
\]
This result demonstrates that the combinatorial block-pillar structure defined purely through the sequence's self-generation property inherently captures the correct dominant Pisot growth factor $\alpha$. This factor $\alpha$ is the same Pisot root that governs the underlying substitution dynamics of the sequence (\cite[pp.~170--171]{BaakeSing2004}) and acts as the inflation multiplier for the associated geometric model set (\cite[Thm.~2]{BaakeSing2004}). Deriving $\alpha$ directly from our recursive framework provides strong internal consistency and validates this framework as a quantitative reflection of the sequence's structural self-similarity. This self-similarity, governed by the Pisot number $\alpha$, underpins the long-range quasicrystalline order of $K(1,3)$, which manifests mathematically as a pure point diffraction spectrum (\cite[Thm.~3]{BaakeSing2004}).

\section{Discussion and Conclusion}
\label{sec:discussion}

\subsection{Summary of Findings}

This paper has introduced and analysed an explicit recursive structure inherent in the Kolakoski sequence $K(1,3)$ over the alphabet $\{1, 3\}$. We defined sequences of blocks $B_n$ and pillars $P_n$ satisfying the mutual recursion $B_{n+1}=B_n \concat P_n \concat B_n$ and $P_{n+1}=\G(R(P_n),3)$, where $\G$ is the run-length generation operator. Crucially, we proved (\Cref{thm:main}) that these recursively defined blocks $B_n$ are precisely the prefixes of $K(1,3)$, and that they satisfy the Kolakoski property directly via $B_{n+1}=\G(R(B_n),1)$.

This block-pillar framework enabled a detailed quantitative analysis. We derived exact recurrence relations for the lengths ($L_n, m_n$) and symbol counts ($c_n, o_n$) of blocks and pillars, leading to the fundamental identity $m_n = L_n - 2c_n = \Nx{3}{B_n} - \Nx{1}{B_n}$ (\Cref{prop:identity}). Analysis of the symbol densities $d_n = c_n/L_n$ and $\delta_n = o_n/m_n$ revealed that if these densities converge to limits $d$ and $\delta$ respectively, and $d \ne 1/2$, then necessarily $d=\delta$ (\Cref{thm:density_limit}). Most significantly, by analyzing the linear system governing the evolution of $(L_n, c_n)$, we demonstrated that the asymptotic growth rate of the block lengths is $\lim_{n\to\infty} L_{n+1}/L_n = \alpha \approx 2.20557$, where $\alpha$ is the dominant Pisot root of $x^3 - 2x^2 - 1 = 0$ (\Cref{sec:growth_rate}).

\subsection{Connection to Substitution Dynamics and Model Sets}

The emergence of the Pisot number $\alpha$ from our purely combinatorial recursion links this framework directly to the known algebraic and geometric properties of $K(1,3)$ (or its mirror $\operatorname{Kol}(3,1)$), particularly as established by Baake and Sing \cite{BaakeSing2004}.

\begin{enumerate}
    \item \textbf{Growth Factor Validation:} Baake and Sing show that the substitution $\sigma$ related to $\operatorname{Kol}(3,1)$ has a characteristic polynomial $x^3 - 2x^2 - 1 = 0$~\cite[p.~170]{BaakeSing2004}. Our derivation of the growth factor $\alpha$ (the dominant root of this polynomial) from the linear system governing $(L_n, c_n)$ (\Cref{sec:growth_rate}) confirms that our recursive structure accurately captures the fundamental scaling behaviour inherent in the substitution dynamics.

    \item \textbf{Density Consistency:} The derived requirement $d=\delta$ (if limits exist) and the numerical convergence observed (\Cref{sec:density_convergence}) towards $d \approx 0.397$ are consistent with the known density of '1's, $d = \rho_1 = (3-\alpha)/2$, calculated from the substitution matrix's eigenvectors \cite[p.~171, Eq.~(6)]{BaakeSing2004}.

    \item \textbf{Structural Analogy:} The $B_n P_n B_n$ decomposition appears to be a combinatorial analogue of the hierarchical self-similarity generated by the underlying Pisot substitution $\sigma$. The recursive generation of blocks and pillars likely mirrors the way the substitution builds increasingly complex structures from simpler ones.

    \item \textbf{Combinatorial Path vs. Geometric Deformation:} Baake and Sing rigorously construct the integer sequence $\operatorname{Kol}(3,1)$ as a \emph{deformed model set}. They start with an ideal geometric model set $\Sigma \operatorname{Kol}(3,1)$ (a Meyer set with points in $\mathbb{Z}[\alpha]$), obtained via a cut-and-project scheme, and then apply an explicit deformation map $\varphi$ to obtain the final sequence with integer spacings \cite[Thms.~2, 3]{BaakeSing2004}. Our recursive framework achieves the same end result – the correctly ordered sequence $K(1,3)$ with integer positions – through a purely combinatorial mechanism. The generation operator $\G$, by directly enforcing the Kolakoski run-length constraints using integer lengths specified by previous terms, inherently produces the final sequence structure. This provides a direct combinatorial path that bypasses the need for intermediate non-integer geometric objects and the explicit deformation step, offering a different but consistent perspective on the sequence's generation.
\end{enumerate}

\subsection{Context within the Kolakoski Family}

The discovery of this explicit and analytically tractable recursive structure in $K(1,3)$ stands in sharp contrast to the behaviour of the classical Kolakoski sequence $K(1,2)$. Despite decades of study, $K(1,2)$ remains deeply enigmatic, lacking any known comparable recursive decomposition, and fundamental properties like the existence and value (presumed 1/2) of its symbol density are still unproven conjectures \cite{Dekking2017}.

Our findings lend support to the general observation that Kolakoski sequences $K(a,b)$ where $a$ and $b$ share the same parity tend to be significantly more structured than those with mixed parity. $K(1,3)$ serves as a prime example of this regularity. It raises the natural question of whether analogous block-pillar structures might exist for other same-parity sequences expected to possess structure, such as $K(a, a+2)$ for odd $a$ (e.g., $K(3,5)$). These sequences are also related to Pisot substitutions \cite[p.~171]{BaakeSing2004}, suggesting similar mechanisms might be at play. Investigating such cases is a promising direction for future research. Note that sequences like $K(2,4)$ are trivially $2 \times K(1,2)$ and thus inherit the complexity of the $K(1,2)$ sequence (e.g., \cite[p.~1]{Dekking2017}). 

\subsection{Conclusion and Future Work}

This work has revealed a novel recursive block-pillar structure governing the Kolakoski sequence $K(1,3)$. This structure is not only computationally effective but also analytically tractable, allowing for the derivation of exact recurrences and, crucially, the sequence's characteristic Pisot growth factor $\alpha$ directly from the combinatorial definition. This provides a new perspective on the sequence's generation, complementing existing approaches based on substitution dynamics and model sets, and highlights the profound structural regularity distinguishing $K(1,3)$ from its classical counterpart $K(1,2)$.

Several questions remain open:
\begin{itemize}
    \item What is the deeper structural reason or symmetry that dictates the specific $B_{n+1} = B_n P_n B_n$ form of the decomposition?
    \item Can the steps of the block-pillar recursion ($B_n \to B_{n+1}$, $P_n \to P_{n+1}$) be explicitly mapped onto the substitution $\sigma$ acting on its alphabet $\{A,B,C\}$~\cite{BaakeSing2004}, or perhaps related to the iterated function system (IFS) associated with the sequence?
    \item Does an analogous, verifiable block-pillar recursion exist for other structured Kolakoski sequences, such as $K(3,5)$ (cf. \cite[p.~171]{BaakeSing2004})? 
\end{itemize}

Answering these questions could further illuminate the mechanisms underlying order and complexity in this fascinating family of self-generating sequences.

\bigskip 
\noindent\textbf{Acknowledgements.}

\noindent
[Placeholder Acknowledgement Awaiting Permission]

\noindent
The author also gratefully acknowledges \texttt{u/LemonadeTsunami} for observing the recursive structure within $K(1,3)$ in the \href{https://www.reddit.com/r/math/comments/1k0vu58/repetetive_pattern_in_kolakoski_sequence_13/}{Reddit post}~\cite{LemonadeTsunamiReddit2023} titled \emph{“Repetitive pattern in Kolakoski sequence \{1,3\}”}. This insightful observation served as the catalyst for the present formal analysis.

\section*{Author's Note}
Many thanks to Professor [*awaiting permission for acknowledgment] for pointing out the critical error in the frequency in the previous version, I hope this has now been fully rectified.

I'm a second-year undergraduate in economics at the University of Bristol. Most of my research so far has been in information theory and statistics, this paper is a bit of a departure. I wanted to try something more combinatorial and structural, partly as a challenge to test myself, and partly out of curiosity about recursive sequences and their hidden symmetries.

This note was written independently, without formal supervision or institutional support. I used AI tools (predominantly Gemini Pro 2.5 and o4-mini-high) to help with constantly checking rigour and consistency, polishing proofs, and sharpening the presentation. All the mathematics, definitions, and results are my own, but the writing process benefited from the kind of iterative back-and-forth that these tools make possible.

Though I’m still learning, I’ve tried to keep the exposition precise, readable, and honest to the sequence's underlying logic. If you have thoughts, corrections, or suggestions (especially if you work on automatic sequences or symbolic dynamics) I’d be very grateful to hear them and full accreditation for any feedback implemented will be added in future versions. 

\bibliographystyle{plainurl} 
\bibliography{references} 

\end{document}